\documentclass[11pt]{article}
\usepackage[margin=1in]{geometry}
\usepackage{enumerate}
\usepackage{url}
\usepackage{cite}

\usepackage{amsmath,amsthm,amssymb}

\title{\textbf{Spaces of maps between real algebraic varieties}}
\author{Wojciech Kucharz}
\date{}

\theoremstyle{plain}
\newtheorem{theorem}{Theorem}[section]
\newtheorem{lemma}[theorem]{Lemma}

\newtheorem{corollary}[theorem]{Corollary}

\theoremstyle{definition}
\newtheorem{definition}[theorem]{Definition}
\newtheorem{example}[theorem]{Example}

\newcounter{thmlist}
\counterwithin*{thmlist}{theorem}

\newcommand{\R}{\mathbb{R}}
\newcommand{\e}{\varepsilon}
\newcommand{\C}{\mathcal{C}}

\begin{document}

\maketitle

\paragraph{Abstract.} Given two real algebraic varieties $X$ and $Y$, we denote by $\mathcal{R}(X,Y)$ the set of all regular maps from $X$ to $Y$. The set $\mathcal{R}(X,Y)$ is regarded as a topological subspace of the space $\C(X,Y)$ of all continuous maps from $X$ to $Y$ endowed with the compact-open topology. We prove, in a much more general setting than previously considered, that each path component of $\C(X,Y)$ contains at most one path component of $\mathcal{R}(X,Y)$, and for every positive integer $k$ the inclusion map $\mathcal{R}(X,Y)\hookrightarrow\C(X,Y)$ induces an isomorphism between the $k$th homotopy groups of the corresponding path components. We also identify several cases where this inclusion map is a weak homotopy equivalence.



\section{Introduction}

Throughout this paper, we use the term \emph{real algebraic variety} to mean a ringed space with structure sheaf of $\R$-algebras of $\R$-valued functions, which is isomorphic to a Zariski locally closed subset of real projective space $\mathbb{P}^n(\R)$, for some $n$, endowed with the Zariski topology and the sheaf of regular functions. This terminology is compatible with that used in the monographs \cite{4,17} containing a detailed treatment of real algebraic geometry. Morphisms of real algebraic varieties are called \emph{regular maps}.

Recall that each real algebraic variety in the sense used here is actually \emph{affine}, i.e. isomorphic to an algebraic set in $\R^n$ for some $n$ \cite[Proposition 3.2.10 and Theorem 3.4.4]{4} or \cite[Proposition 1.3.11]{17}. This fact allows us to describe regular functions and regular maps in a simple way. To be precise, let $X\subset \R^n$ and $Y\subset \R^m$ be algebraic sets, and let $U\subset X$ be a Zariski open set. A function $\varphi:U\rightarrow\R$ is regular if and only if there exist two polynomial functions $P,Q:\R^n\rightarrow \R$ such that 
$$Q(x)\neq 0\,\textnormal{ and }\, \varphi(x)=\frac{P(x)}{Q(x)}\,\textnormal{ for all }\, x\in U$$
[4, p. 62] or \cite[p. 14]{17}. A map $f=(f_1,\ldots, f_m):U\rightarrow Y\subset\R^m$ is regular if and only if its components $f_1$,\ldots, $f_m$ are regular functions.

Every real algebraic variety is also equipped with the Euclidean topology determined by the usual metric on $\R$. Unless explicitly stated otherwise, all topological notions relating to real algebraic varieties will refer to the Euclidean topology.

Given two real algebraic varieties $X$ and $Y$, we denote by $\mathcal{R}(X,Y)$ the set of all regular maps from $X$ to $Y$. The set $\mathcal{R}(X,Y)$ is regarded as a topological subspace of the space $\C(X,Y)$ of all continuous maps from $X$ to $Y$ endowed with the compact-open topology. We will show that the topological properties of the spaces $\mathcal{R}(X,Y)$ and $\C(X,Y)$ are closely related under reasonable assumptions on the target variety $Y$.

\begin{definition}
\label{1.1}
    A real algebraic variety $Y$ is said to be \emph{uniformly retract rational} if for every point $y\in Y$ there exist a Zariski open neighborhood $V$ of $y$ in $Y$, a Zariski open set $W$ in $\R^n$ (for some $n$ depending on $y$), and two regular maps $V\rightarrow W\rightarrow V$ whose composite is the identity on $V$.
\end{definition}

This definition, introduced by Banecki \cite{2}, is a variant of the standard notion of retract rational variety that has been studied since the 1980s.

As usual, given a pointed topological space $(T,t_0)$, we write $\pi_k(T,t_0)$ to denote its $k$th homotopy group, where $k$ is an arbitrary positive integer; in addition, we denote by $\pi_0(T,t_0)$ the set of all path components of $T$.

Our main result is as follows.

\begin{theorem}
\label{1.2}
    Let $Y$ be a uniformly retract rational real algebraic variety. Let $X$ be a compact real algebaric variety and let 
    $$i:\mathcal{R}(X,Y)\hookrightarrow\C(X,Y)$$
    be the inclusion map. Then for every regular map $f\in \mathcal{R}(X,Y)$, the induced map 
    $$i_{*}:\pi_k(\mathcal{R}(X,Y),f)\rightarrow \pi_k(\C(X,Y),f)$$
    is injective for $k=0$ and a group isomorphism for $k\geq 1$.
\end{theorem}

It should be noted that Theorem \ref{1.2} for $k=0$ is included in \cite[Corollary 1.6]{2}. The following conclusion is a direct consequence of Theorem \ref{1.2}.

\begin{corollary}
\label{1.3}
    Under the same assumptions as in Theorem \ref{1.2}, the inclusion map 
    $$i:\mathcal{R}(X,Y)\hookrightarrow \C(X,Y)$$
    is a weak homotopy equivalence if and only if every continuous map from $X$ to $Y$ is homotopic to a regular map.
\end{corollary}

Recall that an $n$-dimensional real algebraic variety is said to be \emph{uniformly rational} if each of its points has a Zariski open neighborhood biregularly isomorphic to a Zariski open subset of $\R^n$. Obviously, every uniformly rational variety is uniformly retract rational. The converse is not true in general. A counterexample is based on a rather sophisticated construction given in \cite{3}. The complex algebraic hypersurface $V$ in $\mathbb{C}^4$ described in \cite[p.315, Example 2]{3} is not rational. Since $V$ is defined over $\R$, the real algebraic hypersurface $V(\R):=V\cap \R^4$ in $\R^4$ is not rational. On the other hand, by \cite[p.299, Th\'{e}or\`{e}me 1']{3} there exists a nonempty Zariski open subset of $V(\R)\times \R^3$ that is biregularly isomorphic to a Zariski open subset of $\R^6$. Thus, some nonempty Zariski open subset $Y$ of $V(\R)$ is a uniformly retract rational real algebraic variety. Clearly, $Y$ is not uniformly rational. I am grateful to Olivier Benoist for bringing the relevant results of \cite{3} to my attention.

The following example illustrates the scope of applicability of Theorem \ref{1.2}.

\begin{example}
\label{1.4}
    Here are some uniformly rational varieties of interest in real algebraic geometry.
    \begin{enumerate}
        \item[(1)] (\textbf{Grassmannians}). Let $\mathbb{F}$ stand for $\R$, $\mathbb{C}$ or $\mathbb{H}$, where $\mathbb{H}$ is the skew field of quaternions. Denote by $\mathbb{G}_r(\mathbb{F}^n)$ the Grassmannian of $r$-dimensional $\mathbb{F}$-vector subspaces of $\mathbb{F}^n$, $0\leq r\leq n$. Then $\mathbb{G}_r(\mathbb{F}^n)$ considered as a real algebraic variety is uniformly rational \cite[pp. 77, 73, 352]{4}.
        \item[(2)] (\textbf{Unit spheres}). The unit $n$-sphere 
        $$\mathbb{S}^n:=\{ (x_0,\ldots,x_n)\in \R^{n+1}:\, x_0^2+\ldots+x_n^2=1 \}$$
        is uniformly rational because $\mathbb{S}^n$ with any point removed is biregularly isomorphic to $\R^n$ via an appropriate stereographic projection.
        \item[(3)] (\textbf{Some linear real algebraic groups}). A \emph{linear real algebraic group} is a Zariski closed subgroup of the general linear group $GL(n,\R)$, for some $n$. It is an open question whether every such group is a uniformly rational real algebraic variety. However, this is the case if $G$ is one of the groups $SO(n)$, $O(n)$, $SU(n)$ or $U(n)$ ($SU(n)$ and $U(n)$ are linear real algebraic groups, viewed as subgroups of $GL(2n,\R)$). Indeed, the Cayley transform
        $$A\mapsto (I-A)(I+A)^{-1}$$
        defines a biregular isomorphism from the tangent space $T_IG$ to $G$ at the identity matrix $I$ onto the Zariski open neighborhood
        $$G_I:=\{ Q\in G:\, \textnormal{det}(I+Q)\neq 0 \}$$
        of $I$ in $G$; here $T_ISO(n)=T_IO(n)$ is the $\R$-vector space of all real skew-symmetric $n\times n$ matrices, $T_IU(n)$ is the $\R$-vector space of all skew-Hermitian $n\times n$ matrices, and $T_ISU(n)$ is the $\R$-vector subspace of $T_IU(n)$ consisting of all matrices with trace $0$. The uniform rationality of $G$ follows since for every $P\in G$ the set $\{ PQ:\, Q\in G_I \}$ is a Zariski open neighborhood of $P$ biregularly isomorphic to $G_I$.
        \item[(4)] (\textbf{Real form of complex groups}). By Chevalley's theorem \cite[Corollary 2]{11}, every linear complex algebraic group $G\subset GL(n,\mathbb{C})$ has the property that its irreducible component containing the identity matrix is a complex rational variety. Thus, the image of $G$ under the standard embedding 
        $$GL(n,\mathbb{C})\rightarrow GL(2n,\R)$$
        is a linear real algebraic group which is a uniformly rational variety.
        \item[(5)] (\textbf{Blow-ups of uniformly rational varieties}). If $Y$ is a uniformly rational real algebraic variety and $Z\subset Y$ is a nonsingular Zariski closed subvariety, then the blow-up of $Y$ with center $Z$ is a uniformly rational variety. The proof given in \cite[p. 885]{14} and \cite{9} in a complex setting also works for real algebraic varieties.
    \end{enumerate}
\end{example}

By Example \ref{1.4}(1), both Theorem \ref{1.2} and Corollary \ref{1.3} hold for $Y=\mathbb{G}_r(\mathbb{F}^n)$, in which case these results are included in \cite[Theorem 1.1 and Corollary 1.2]{7}. The proof given in \cite{7} depends heavily on the theory of algebraic vector bundles and is not applicable to the present, more general context.

According to \cite[Example 1.3]{7}, for every nonnegative integer $m$ the inclusion map 
$$\mathcal{R}(\mathbb{S}^m,\mathbb{G}_r(\mathbb{F}^n))\hookrightarrow \C(\mathbb{S}^m,\mathbb{G}_r(\mathbb{F}^m))$$
is a weak homotopy equivalence. As shown below, the latter result can be extended to several new cases.

The next conclusion follows immediately from Corollary \ref{1.3} and Example \ref{1.4}(2).

\begin{corollary}
\label{1.5}
    Let $(m,n)$ be a pair of nonnegative integers. Then the inclusion map
    $$\mathcal{R}(\mathbb{S}^m,\mathbb{S}^n)\hookrightarrow \C(\mathbb{S}^m,\mathbb{S}^n)$$
    is a weak homotopy equivalence if and only if every continuous map from $\mathbb{S}^m$ to $\mathbb{S}^n$ is homotopic to a regular map.
\end{corollary}

It is expected that every continuous map from $\mathbb{S}^m$ to $\mathbb{S}^n$ is homotopic to a regular map for all pairs $(m,n)$, but currently this is only proved for some pairs, e.g. those with $n\in \{ 1,2,4 \}$ \cite{5}, $n\in\{ 3,7 \}$ \cite{1}, and $m-n\leq 5$ \cite{1,6,18,19} (see also \cite{1,4} for other cases).

\begin{corollary}
\label{1.6}
    If $G$ is one of the groups $SO(n)$, $O(n)$, $SU(n)$ or $U(n)$, then for every nonnegative integer $m$ the inclusion map 
    $$\mathcal{R}(\mathbb{S}^m,G)\hookrightarrow \C(\mathbb{S}^m,G)$$
    is a weak homotopy equivalence.
\end{corollary}
\begin{proof}
    Since by \cite[Theorem 1.5]{1} every continuous map from $\mathbb{S}^m$ to $G$ is homotopic to a regular map, it is sufficient to invoke Corollary \ref{1.3} and Example \ref{1.4}(3).
\end{proof}

Next, we show that in Examples \ref{1.5} and \ref{1.6} the unit sphere $\mathbb{S}^m$ cannot be replaced by an arbitrary compact nonsingular real algebraic variety.

\begin{example}
\label{1.7}
    Let $Y$ be a compact nonsingular real algebraic variety of dimension at least 1. By \cite[Theorem 1.1]{8}, we can choose a compact connected nonsingular real algebraic variety $X$ and a continuous map $f:X\rightarrow Y$ such that $f$ is not homotopic to any regular map from $X$ to $Y$.
\end{example}

Note that in Theorem \ref{1.2} the assumption on the target variety $Y$ cannot be weakened too much.

\begin{example}
\label{1.8}
    There exists a nonsingular real algebraic variety $Y$ diffeomorphic to $\mathbb{S}^2$ such that the homotopy groups 
    $$\pi_k(\mathcal{R}(\mathbb{S}^3,Y),f) \textnormal{ and } \pi_k(\C(\mathbb{S}^3,Y),f)$$
    are nonisomorphic for all $f\in \mathcal{R}(\mathbb{S}^3,Y)$ and all $k\geq 1$. Such a variety $Y$ can be constructed as follows.

    Every regular map from $\mathbb{S}^1$ to the curve 
    $$C:=\{ (x,y)\in\R^2:\, x^4+y^4=1 \}$$
    is constant (see e.g. \cite[Example 5.2.6]{17}), so for every integer $m\geq 1$ every regular map $\mathbb{S}^m\rightarrow C$ is constant. We can choose a nonsingular Zariski closed subvariety $Y$ of the Cartesian power $C^3$ that is diffeomorphic to $\mathbb{S}^2$ (this claim follows from \cite[Theorem 12.4.11]{4} since there exists a null homotopic $\C^{\infty}$ embedding of $\mathbb{S}^2$ into $C^3$). Obviously, every regular map $\mathbb{S}^m\rightarrow Y$ is constant, and hence the space $\mathcal{R}(\mathbb{S}^m,Y)$ is homeomorphic to $\mathbb{S}^2$. In particular, for $m=3$, we get
    $$\pi_k(\mathcal{R}(\mathbb{S}^3,Y),f)\cong\pi_k(\mathbb{S}^2)$$
    for all $f\in\mathcal{R}(\mathbb{S}^3,Y)$ and all $k\geq 1$. On the other hand, according to \cite[p. 48, Corollary]{15},
    $$\pi_k(\mathcal{C}(\mathbb{S}^3,Y),f)\cong \pi_k(\mathbb{S}^2)\oplus \pi_{k+2}(\mathbb{S}^2)$$
    for all $f\in\C(\mathbb{S}^3,Y)$ and all $k\geq 1$. Furthermore, the homotopy group $\pi_j(\mathbb{S}^2)$ is nontrivial for all $j\geq 2$ by \cite[Theorem 1.1]{16}. Thus $Y$ has the required properties.
\end{example}

In Example \ref{1.8}, we took advantage of the fact that the set of regular maps in the space of continuous maps can be very small. The latter phenomenon is quite common, as demonstrated by Ghiloni in \cite{13}.

Theorem \ref{1.2} follows from Theorem \ref{2.5} proved in Section \ref{2}. These  two results can be seen as real algebraic counterparts of some results of Gromov \cite{14} and Forstneri\v{c} \cite[Corollary 5.5.6]{12} on the relationship between the space of holomorphic maps and the space of continuous maps. The real algebraic case requires innovative methods.

\section{Properties of uniformly retract rational varieties}
\label{2}

Theorem \ref{1.2} is a special case of Theorem \ref{2.5} below. To prove Theorem \ref{2.5}, we need some highly nontrivial properties of uniformly retract rational varieties. These are taken from Banecki's paper \cite{2}.

\begin{definition}[see {\cite[Definition 2.1]{2}}]
\label{2.1}
    Let $Y$ be a real algebraic variety. A \emph{strong dominating spray} $(n,M,N,\sigma,\tau)$ for $Y$ consists of a nonnegative integer $n$, a Zariski open set $M\subset Y\times \R^n$ containing $Y\times\{0\}$, a Zariski open set $N\subset Y\times Y$ containing the diagonal $\triangle_Y$ of $Y\times Y$, and two regular maps 
    $$\sigma:M\rightarrow Y,\quad \tau:N\rightarrow\R^n$$
    satisfying the following conditions:
    \begin{enumerate}
        \item[(a)] $(y,\tau(y,z))\in M$ and $\sigma(y,\tau(y,z))=z$ for all $(y,z)\in N$,
        \item[(b)] $\tau(y,y)=0$ for all $y\in Y$.
    \end{enumerate}
\end{definition}

Note that conditions (a), (b) yield
$$\sigma(y,0)=y\, \textnormal{ for all } y\in Y.$$

\begin{theorem}[see {\cite[Theorem 2.5]{2}}]
\label{2.2}
    Every uniformly retract rational real algebraic variety admits a strong dominating spray.
\end{theorem}

Prior to formulating the next theorem, some preparation is necessary. Let $X$, $Y$ be real algebraic varieties, $S$ an arbitrary subset of $X$, and $Z$ the Zariski closure of $S$ in $X$.

A map $\varphi:S\rightarrow Y$ is said to be \emph{regular} if there exists a regular (in the usual sense) map $\varphi_0:Z_0\rightarrow Y$ defined on a Zariski open neighborhood $Z_0$ of $S$ in $Z$ such that $\varphi_0|_S=\varphi$.

Denote by $\mathcal{R}(S,Y)$ the set of all regular maps from $S$ to $Y$. The set $\mathcal{R}(S,Y)$ is regarded as a topological subspace of the space $\C(S,Y)$ of all continuous maps from $S$ to $Y$ endowed with the compact-open topology. We say that a continuous map $f:S\rightarrow Y$ can be \emph{approximated in the compact-open topology by regular maps} if every neighborhood of $f$ in $\C(S,Y)$ contains a regular map.

\begin{theorem}[see {\cite[Theorem 1.3]{2}}]
\label{2.3}
    Let $Y$ be a uniformly retract rational real algebraic variety. Let $X$ be a real algebraic variety, $Z$ a Zariski closed subvariety of $X$, and $f:K\rightarrow Y$ a continuous map defined on a compact semialgebraic subset $K$ of $X$. Assume that $f$ is homotopic to a regular map from $K$ to $Y$, and the restriction $f|_{K\cap Z}$ is a regular map. Then $f$ can be approximated in the compact-open topology by regular maps from $K$ to $Y$ whose restrictions to $K\cap Z$ are equal to $f|_{K\cap Z}$.
\end{theorem}

The following lemma can be easily deduced from Theorems \ref{2.2} and \ref{2.3}.

\begin{lemma}
\label{2.4}
    Let $Y$ be a uniformly retract rational real algebraic variety. Let $X$ be a real algebraic variety and let $K$ be a compact semialgebraic subset of $X$. Let $l$ be a nonnegative integer, $P$ a compact semialgebraic subset of $\R^l$, and $Q$, $Q_0$ two Zariski closed subvarieties of $\R^l$ satisfying $Q_0\subset Q\subset P$. Let $\varphi:K\times P\rightarrow Y$ be a continuous map with the following properties:
    \begin{enumerate}
        \item[(i)] $\varphi$ is homotopic to a regular map $K\times P\rightarrow Y$,
        \item[(ii)] the restriction $\varphi|_{K\times Q_0}$ is a regular map,
        \item[(iii)] the map $\varphi(\cdot, q):K\rightarrow Y$, $x\mapsto \varphi(x,q)$ is regular for all $q\in Q$.
    \end{enumerate}
    Then there exist a regular map $\psi:K\times P\rightarrow Y$ and a continuous map $\Phi:K\times P\times [0,1]\rightarrow Y$ such that the following hold:
    \begin{enumerate}
        \item[(1)] $\psi|_{K\times Q_0}=\varphi|_{K\times Q_0}$,
        \item[(2)] $\Phi$ is a homotopy from $\varphi$ to $\psi$ relative to $K\times Q_0$,
        \item[(3)] the map $\Phi(\cdot,q,t):K\rightarrow Y$, $x\mapsto \Phi(x,q,t)$ is regular for all $(q,t)\in Q\times [0,1]$.
    \end{enumerate}
\end{lemma}
\begin{proof}
    According to Theorem \ref{2.2}, $Y$ admits a strong dominating spray $(n,M,N,\sigma,\tau)$. Furthermore, by (i), (ii) and Theorem \ref{2.3}, there exists a regular map $\psi:K\times P\rightarrow Y$, arbitrarily close to $\varphi$ in the compact-open topology, with $\psi|_{K\times Q_0}=\varphi|_{K\times Q_0}$. In particular, (1) holds. Thus, in view of (iii) and Definition \ref{2.1}, if $\psi$ is sufficiently close to $\varphi$, we obtain a well-defined continuous map 
    $$\Phi:K\times P\times [0,1]\rightarrow Y,\, \Phi(x,p,t):=\sigma(\varphi(x,p),t\tau(\varphi(x,p),\psi(x,p)))$$
    satisfying (2) and (3).
\end{proof}

It is convenient to adopt the following convention. If $(T,t_0)$ is a pointed topological space and $k$ is a nonnegative integer, then each element of $\pi_k(T,t_0)$ is considered to be the homotopy class represented by a pointed continuous map $(\mathbb{S}^k,s_0)\rightarrow (T,t_0)$, where $s_0=(1,0,\ldots,0)\in\mathbb{S}^k$ (in particular, $s_0=1\in S^0=\{ 1,-1 \}$ for $k=0$). The neutral element of $\pi_k(T,t_0)$ is the homotopy class of the constant map sending $\mathbb{S}^k$ to $t_0$. We will regard $\mathbb{S}^k$ as the boundary of the Euclidean unit closed ball $\mathbb{B}^{k+1}$ in $\R^{k+1}$.

The main result of this section is the following generalization of Theorem \ref{1.2}.

\begin{theorem}
\label{2.5}
    Let $Y$ be a uniformly retract rational real algebraic variety. Let $X$ be a real algebraic variety, $K$ a compact semialgebraic subset of $X$, and
    $$i:\mathcal{R}(K,Y)\hookrightarrow\C(K,Y)$$
    the inclusion map. Then for every regular map $f\in \mathcal{R}(K,Y)$, the induced map
    $$i_{*}:\pi_k(\mathcal{R}(K,Y),f)\rightarrow \pi_k(\C(K,Y),f)$$
    is injective for $k=0$ and a group isomorphism for $k\geq 1$.
\end{theorem}
\begin{proof}
    For simplicity, we set $\mathcal{R}:=\mathcal{R}(K,Y)$ and $\C:=\C(K,Y)$. The proof consists of two steps in which $f\in \mathcal{R}$.
    \vspace{5pt}

    \noindent
    \emph{Step 1.} The map
    $$i_{*}:\pi_k(\mathcal{R},f)\rightarrow \pi_k(\C,f)$$
    is injective for $k\geq 0$.

    \begin{proof}[Proof of Step 1]
        Let $\alpha:K\times\mathbb{S}^k\rightarrow Y$ be a continuous map such that $\alpha(\cdot, u):K\rightarrow Y$ is a regular map for all $u\in\mathbb{S}^k$, and $\alpha(\cdot,s_0)=f$. Then $\alpha_{\mathcal{R}}:(\mathbb{S}^k,s_0)\rightarrow (\mathcal{R},f)$ defined by the formula 
    $$\alpha_{\mathcal{R}}(u)(x)=\alpha(x,u)\quad \textnormal{for all } x\in K \textnormal{ and } u\in \mathbb{S}^k$$
    is a continuous map of pointed spaces. Actually, each continuous map of pointed spaces $(\mathbb{S}^k,s_0)\rightarrow (\mathcal{R},f)$ is of the form $\alpha_{\mathcal{R}}$ for some $\alpha$ as above. Let $[\alpha_{\mathcal{R}}]$ denote the homotopy class in $\pi_k(\mathcal{R},f)$ represented by $\alpha_{\mathcal{R}}$ and assume that $i_{*}([\alpha_{\mathcal{R}}])$ is the neutral element of $\pi_k(\C,f)$. Our goal is to prove that $[\alpha_{\mathcal{R}}]$ is the neutral element of $\pi_k(\mathcal{R},f)$.

    By the assumption, $i\circ\alpha_{\mathcal{R}}$ can be extended to a continuous map from $\mathbb{B}^{k+1}$ to $\C$, so there exists a continuous extension $\varphi:K\times \mathbb{B}^{k+1}\rightarrow Y$ of $\alpha$. In particular, $\varphi(\cdot,u)=\alpha(\cdot,u)$ for all $u\in\mathbb{S}^k$. Moreover,
    $$K\times \mathbb{B}^{k+1}\times [0,1]\rightarrow Y,\quad (x,u,t)\mapsto \varphi(x,(1-t)u+ts_0)$$
    is a homotopy from $\varphi$ to the regular map 
    $$K\times \mathbb{B}^{k+1}\rightarrow Y,\quad (x,u)\mapsto f(x).$$
    Thus, according to Lemma \ref{2.4} (with $P=\mathbb{B}^{k+1}$, $Q=\mathbb{S}^k$, $Q_0=\{ s_0 \}$), there exist a regular map $\psi:K\times \mathbb{B}^{k+1}\rightarrow Y$ and a continuous map $\Phi:K\times \mathbb{B}^{k+1}\times [0,1]\rightarrow Y$ such that 
    \begin{itemize}
        \item $\psi(\cdot,s_0)=\varphi(\cdot,s_0)=f$,
        \item $\Phi$ is a homotopy from $\varphi$ to $\psi$ relative to $K\times \{s_0\}$,
        \item $\Phi(\cdot,u,t)$ is a regular map for all $(u,t)\in\mathbb{S}^k\times [0,1]$.
    \end{itemize}
    Now, let $\beta:=\psi|_{K\times \mathbb{S}^k}$. Since $\psi$ is a regular map, the map $\beta_{\mathcal{R}}:(\mathbb{S}^k,s_0)\rightarrow (\mathcal{R},f)$ has a continuous extension $\psi_{\mathcal{R}}:\mathbb{B}^{k+1}\rightarrow\mathcal{R}$ defined by the formula $\psi_{\mathcal{R}}(u)(x)=\psi(x,u)$ for all $x\in K$ and $u\in \mathbb{B}^{k+1}$, so the homotopy class $[\beta_{\mathcal{R}}]$ is the neutral element of $\pi_k(\mathcal{R},f)$. On the other hand, $\alpha_{\mathcal{R}}$ and $\beta_{\mathcal{R}}$ are homotopic as pointed maps via the homotopy $A:\mathbb{S}^k\times [0,1]\rightarrow \mathcal{R}$ defined by the formula 
    $$A(u,t)(x):=\Phi(x,u,t)\quad \textnormal{for all } x\in K \textnormal{ and } (u,t)\in\mathbb{S}^k\times [0,1].$$
    Thus the homotopy class $[\alpha_{\mathcal{R}}]$ is indeed the neutral element of $\pi_k(\mathcal{R},f)$ as we wanted to show.
    \end{proof}

    \noindent
    \emph{Step 2.} The map 
    $$i_{*}:\pi_k(\mathcal{R},f)\rightarrow\pi_k(\C,f)$$
    is surjective for $k\geq 1$.
    \vspace{10pt}

    \noindent
    \emph{Proof of Step 2.} Let $\mu:K\times \mathbb{S}^k\rightarrow Y$ be a continuous map with $\mu(\cdot,s_0)=f$. Then $\mu_{\C}:(\mathbb{S}^k,s_0)\rightarrow (\C,f)$ defined by the formula
    $$\mu_{\C}(u)(x)=\mu(x,u)\quad \textnormal{for all } x\in K \textnormal{ and } u\in\mathbb{S}^k$$
    is a continuous map of pointed spaces. Obviously, each continuous map of pointed spaces $(\mathbb{S}^k,s_0)\rightarrow (\C,f)$ is of the form $\mu_{\C}$ for some $\mu$ as above. Our aim is to prove that the homotopy class $[\mu_{\C}]\in\pi_k(\C,f)$ represented by $\mu_{\C}$ belongs to the image of $i_{*}$.

    Setting 
    $$\mathbb{B}^k_+:=\{ (u_0,\ldots, u_k)\in\mathbb{S}^k:\, u_k\geq 0 \},$$
    $$\mathbb{B}^k_-:=\{ (u_0,\ldots, u_k)\in \mathbb{S}^k:\, u_k\leq 0 \},$$
    we get $\mathbb{B}^k_+\cap\mathbb{B}^k_-=\mathbb{S}^{k-1}$, where $\mathbb{S}^{k-1}$ is identified with the subset 
    $$\{ (u_0,\ldots, u_k)\in \mathbb{S}^k:\, u_k=0 \}$$
    of $\mathbb{S}^k$. In addition, the canonical projection
    $$\pi:\mathbb{S}^k\rightarrow\mathbb{B}^k,\quad (u_0,\ldots, u_{k-1},u_k)\mapsto (u_0,\ldots, u_{k-1})$$
    has the property that the restrictions $\pi|_{\mathbb{B}^k_+}$ and $\pi|_{\mathbb{B}^k_-}$ are homeomorphisms.

    Note that $\mu|_{K\times \{s_0\}}$ is a regular map. Moreover, the restriction $\mu|_{K\times \mathbb{B}^k_+}$ is homotopic to the regular map 
    $$K\times \mathbb{B}^k_+\rightarrow Y,\quad (x,u)\mapsto f(x)$$
    via the homotopy
    $$K\times \mathbb{B}^k_+\times [0,1]\rightarrow Y,\quad (x,u,t)\mapsto \mu(x,(\pi|_{\mathbb{B}^k_+})^{-1}((1-t)\pi(u)+t\pi(s_0))).$$
    Thus, by Theorem \ref{2.3}, there exists a regular map $\mu_+:K\times \mathbb{B}^k_+\rightarrow Y$, arbitrarily close to $\mu|_{K\times \mathbb{B}^k_+}$ in the compact-open topology, such that $\mu_+|_{K\times \{s_0\}}=\mu|_{K\times \{s_0\}}$. In view of Lemma \ref{2.6} below, we can choose a continuous map $\theta:K\times \mathbb{S}^k\rightarrow Y$ that is close to $\mu$ in the compact-open topology and satisfies $\theta|_{K\times \mathbb{B}^k_+}=\mu_+$; in fact, $\theta$ can be chosen arbitrarily close to $\mu$ if $\mu_+$ is sufficiently close to $\mu|_{K\times \mathbb{B}^k_+}$. By construction, $\theta|_{K\times \mathbb{S}^{n-1}}=\mu_+|_{K\times \mathbb{S}^{n-1}}$, and the restriction $\theta|_{K\times \mathbb{B}^k_-}$ is homotopic to the regular map
    $$K\times \mathbb{B}^k_-\rightarrow Y,\quad (x,u)\mapsto f(x)$$
    via the homotopy
    $$K\times \mathbb{B}^k_-\times [0,1]\rightarrow Y,\quad (x,u,t)\mapsto \theta(x,(\pi|_{\mathbb{B}^k_-})^{-1}((1-t)\pi(u)+t\pi(s_0))).$$
    Hence, according to Theorem \ref{2.3}, there exists a regular map $\mu_-:K\times \mathbb{B}^k_-\rightarrow Y$, arbitrarily close to $\theta|_{K\times \mathbb{B}^k_-}$ in the compact-open topology, such that $\mu_-|_{K\times \mathbb{S}^{k-1}}=\theta|_{K\times \mathbb{S}^{k-1}}$. Since $\mu_+$ and $\mu_-$ agree on $K\times \mathbb{S}^{k-1}$, we get a continuous map $\lambda:K\times \mathbb{S}^k\rightarrow Y$ that is equal to $\mu_+$ and $K\times \mathbb{B}^k_+$ and equal to $\mu_-$ in $K\times \mathbb{B}^k_-$. In particular, $\lambda|_{K\times \{s_0\}}=\mu|_{K\times \{s_0\}}$ and $\lambda(\cdot, u):K\rightarrow Y$ is a regular map for all $u\in\mathbb{S}^k$. The construction yields $\lambda$ that is as close to $\mu$ in the compact-open topology as desired. Therefore, by Lemma \ref{2.7} below, $\lambda$ and $\mu$ are homotopic relative to $K\times \{s_0\}$, which implies $[\lambda_{\C}]=[\mu_{\C}]$ in $\pi_k(\C,f)$. The proof of Step 2 is complete since $i_{*}([\lambda_{\mathcal{R}}])=[\lambda_{\C}]$. This concludes the proof of Theorem \ref{2.5}.
\end{proof}

In the proof of Theorem \ref{2.5}, we used two simple topological facts given below. We include their proofs for completeness.

\begin{lemma}
\label{2.6}
    Let $C$ be a Hausdorff compact space, $Y$ a real algebraic variety, $f:C\rightarrow Y$ a continuous map, \textnormal{dist} a metric on $Y$ that induces the Euclidean topology, $\e>0$ a constant, and $D$ a closed subset of $C$. For these data, there exists a constant $\delta>0$ with the following property: If $\varphi:D\rightarrow Y$ is a continuous map satisfying 
    $$\textnormal{dist}(f(x),\varphi(x))<\delta\quad \textnormal{for all } x\in D,$$
    then there exists a continuous map $g:C\rightarrow Y$ such that $g|_{D}=\varphi$ and 
    $$\textnormal{dist}(f(x),g(x))<\e\quad \textnormal{for all } x\in C.$$
\end{lemma}
\begin{proof}
    We may assume that $Y$ is a Zariski closed subvariety of $\R^m$ for some $m$, and the metric dist is induced by a norm $||\cdot||$ on $\R^m$. Since $Y$ is a locally contractible space \cite[Corollary 9.3.7]{4}, by Borsuk's theorem \cite[Theorem E.3]{10} there exists a continuous retraction $r:W\rightarrow Y$ of an open neighborhood $W$ of $Y$ in $\R^m$. Shrinking $W$ if necessary, we get 
    $$||y-r(y)||<\frac{\e}{2}\quad \textnormal{for all } y\in W.$$
    Now, choose a constant $\delta$, with $0<\delta<\frac{\e}{2}$, such that the set 
    $$\{ y\in\R^m:\, ||f(x)-y||<\delta \textnormal{ for some } x\in C \}$$
    is contained in $W$, and let $\varphi:D\rightarrow Y$ be a continuous map satisfying 
    $$||f(x)-\varphi(x)||<\delta\quad \textnormal{for all } x\in D.$$
    According to Tietze's theorem, there exists a continuous extension $\psi:C\rightarrow\R^m$ of $\varphi$. If 
    $$U:=\{ x\in C:\, ||f(x)-\psi(x)||<\delta \},$$
    then $C\setminus U$ and $D$ are disjoint closed subsets of $C$, so by Urysohn's lemma we can choose a continuous function $\alpha:C\rightarrow[0,1]$ such that $\alpha=0$ on $C\setminus U$ and $\alpha=1$ on $D$. Defining the map $h:C\rightarrow \R^m$ by $h:=f+\alpha(\psi-f)$, we get 
    $$||f(x)-h(x)||=\alpha(x)||f(x)-\psi(x)||<\delta\quad \textnormal{for all } x\in C,$$
    so $h(C)$ is contained in $W$ and $g:=r\circ h$ is a well-defined continuous map from $C$ to $Y$. Clearly, $g|_D=\varphi$ and 
    $$||f(x)-g(x)||\leq ||f(x)-h(x)||+||h(x)-r(h(x))||<\delta+\frac{\e}{2}<\e$$
    for all $x\in C$, which completes the proof.
\end{proof}

\begin{lemma}
\label{2.7}
    Let $C$ be a compact space, $Y$ a real algebraic variety, $f:C\rightarrow Y$ a continuous map, and $D$ a subset of $C$. If $g:C\rightarrow Y$ is a continuous map sufficiently close to $f$ in the compact-open topology and if $g|_D=f|_D$, then $f$ is homotopic to $g$ relative to $D$.
\end{lemma}
\begin{proof}
    As in the proof of Lemma \ref{2.6}, we may assume that $Y$ is a Zariski closed subvariety of $\R^m$ and and there exists a continuous retraction $r:W\rightarrow Y$ of an open neighborhood $W$ of $Y$ in $\R^m$. If $g:C\rightarrow Y$ is a continuous map sufficiently close to $f$ in the compact-open topology and if $g|_D=f|_D$, then 
    $$C\times [0,1]\rightarrow Y,\quad (x,t)\mapsto r(f(x)+t(g(x)-f(x)))$$
    is well-defined continuous map, which is the required homotopy.
\end{proof}

We conclude with the following observation.

\begin{corollary}
\label{2.8}
    Under the same assumptions as in Theorem \ref{2.5}, if $K$ is a contractible space, then the spaces $\mathcal{R}(K,Y)$ and $Y$ are weakly homotopy equivalent.
\end{corollary}
\begin{proof}
    Since $K$ is a contractible space, every continuous map from $K$ to $Y$ is homotopic to a constant map, so the inclusion map 
    $$\mathcal{R}(K,Y)\hookrightarrow\C(K,Y)$$
    is a weak homotopy equivalence by Theorem \ref{2.5}. Now it suffices to show that the spaces $\C(K,Y)$ and $Y$ are homotopy equivalent (which is the case if $Y$ is an arbitrary real algebraic variety with no additional requirements). If $x_0$ is a point in $K$, then the map 
    $$\C(K,Y)\rightarrow \C(\{x_0\},Y),\quad f\mapsto f|_{\{x_0\}}$$
    is a homotopy equivalence, while the spaces $\C(\{x_0\},Y)$ and $Y$ are homeomorphic. This completes the proof.
\end{proof}

\paragraph{Acknowledgements.} The author is indebted to Olivier Benoist for useful correspondence. Partial financial support was provided by the National Science Center (Poland) under grant no. 2022/47/B/ST1/00211.

\vspace{20pt}

\noindent
Institute of Mathematics\\
Faculty of Mathematics and Computer Science\\
Jagiellonian University\\
Łojasiewicza 6\\
30-348 Kraków\\
Poland\\
E-mail address: Wojciech.Kucharz@im.uj.edu.pl\\


\begin{thebibliography}{99}

    \bibitem{1} J. Banecki, Algebraic homotopy classes, J. Math. Pures Appl. 187 (2024), 45-57.

    \bibitem{2} J. Banecki, Relative Stone-Weierstrass theorem for mappings between varieties, arXiv: 2408.09233 (2024).

    \bibitem{3} A. Beauville, J.-L. Colliot-Th\'{e}l\`{e}ne, J.-J. Sansuc and P. Swinnerton-Dyer, Vari\'{e}t\'{e}s stablement rationnelles non rationnelles, Ann. of Math. 121 (1985), 283-318.

    \bibitem{4} J. Bochnak, M. Coste and M.-F. Roy, Real Algebraic Geometry, Ergeb. Math. Grenzgeb. Folge 3, vol. 36, Springer, Berlin, 1998.

    \bibitem{5} J. Bochnak and W. Kucharz, Algebraic approximation of mappings into spheres, Michigan Math. J. 34 (1987), 119-125.
    
    \bibitem{6} J. Bochnak and W. Kucharz, Realization of homotopy classes by algebraic mappings, J. Reine Angew. Math. 377 (1987), 159-169.

    \bibitem{7} J. Bochnak and W. Kucharz, The homotopy groups of some spaces of real algebraic morphisms, Bull. London Math. Soc. 25 (1993), 385-392.

    \bibitem{8} J. Bochnak and W. Kucharz, Real algebraic morphisms represent few homotopy classes, Math. Ann. 377 (2007). 909-921.

    \bibitem{9} F. Bogomolov and C. B\"{o}hning, On uniformly rational varieties, Topology, geometry, integrable systems, and mathematical physics, 234, 33-48, American Mathematical Society, 2014.

    \bibitem{10} G. E. Bredon, Topology and Geometry, GTM, vol. 139, Springer, New York, 1993.

    \bibitem{11} C. Chevalley, On algebraic group varieties, J. Math. Soc. Japan 6 (1954), 303-324.

    \bibitem{12} F. Forstneri\v{c}, Stein Manifolds and Holomorphic Mappings: The Homotopy Principle in Complex Analysis, second edition, Ergeb. Math. Grenzgeb. Folge 3, vol. 56, Springer, Cham, 2017.

    \bibitem{13} R. Ghiloni, On the space of morphisms into generic real algebraic varieties, Ann. Sc. Norm. Super. Pisa. Cl. Sci. (5) 5 (2006), 419-438.

    \bibitem{14} M. Gromov, Oka's principle for holomorphic sections of elliptic bundles, J. Amer. Math. Soc. 2 (1989), 851-897.

    \bibitem{15} V. L. Hansen, The homotopy type of the space of maps of a homotopy 3-sphere into the 2-sphere, Pacific J. Math. 76 (1) (1978), 43-49.

    \bibitem{16} S. O. Ivanov, R. Mikhailov and J. Wu, On nontriviality of certain homotopy groups of spheres, Homology, Homotopy Appl. 18 (2) (2016), 337-344.

    \bibitem{17} F. Mangolte, Real Algebraic Varieties, Springer Monographs in Matheatics, Springer International Publishing, 2020.

    \bibitem{18} J. Peng and Z. Tang, Algebraic maps from spheres to spheres, Sci. China Ser. A 42 (1999), 1147-1154.

    \bibitem{19} R. Wood, Polynomial maps from spheres to spheres, Invent. Math. 5 (1968), 163-168.
    
\end{thebibliography}
\end{document}